\documentclass[12pt]{amsart}
\usepackage{amssymb}
\def\erf{\hbox{  erf}}

\newtheorem{defin}{Definition}
\newtheorem{propo}{ Proposition}
\newtheorem{lemme}{Lemma}

\newtheorem{theorem}{Theorem}
\newtheorem{remark}{Remark}
\newtheorem{exe}{Example}

\begin{document}

\title [Dunkl completely monotonic functions]{Dunkl completely monotonic functions \\}%

\author[ J. El Kamel, K. Mehrez]{ Jamel El Kamel \qquad and \qquad Khaled Mehrez }
 \address{Jamel El Kamel. D\'epartement de Math\'ematiques fsm. Monastir 5000, Tunisia.}
 \email{jamel.elkamel@fsm.rnu.tn}
 \address{Khaled Mehrez. D\'epartement de Math\'ematiques IPEIM. Monastir 5000, Tunisia.}
 \email{k.mehrez@yahoo.fr}
\begin{abstract}
We introduce the notion of Dunkl completely monotonic functions on $\left(-\sigma,\sigma\right), \sigma>0$. We establish a restrictive version of the analogue of Schoenberg's theorem in Dunkl setting.
\end{abstract}
\maketitle
\noindent{\it keywords:} Dunkl operators,  Dunkl translation, positive definite functions,  Dunkl positive definite functions, Dunkl completely monotonic functions, The Kummer confluent hypergeometric functions .\\

\noindent MSC (2010) 26A48, 42A82, 33C52, 42A38\\

\section{Introduction}
\noindent Positive  definite and completely monotonic functions play an important role in harmonic
analysis, for examples, in theory of scattered data interpolation,
probability theory, potential theory. The most important
facts about positive definite functions are the connection between 
positive definite and completely monotonic functions.\\
In classical analysis a complex valued continuous function $f$ is said positive definite (resp. strictly positive definite) on $\mathbb{R}$, if for every distinct real numbers $x_{1},x_{2},...,x_{n}$ and every complex numbers $z_{1},z_{2},...,z_{n}$ not all zero, the inequality 
\[\sum_{j=1}^{n}\sum_{k=1}^{n}z_{j}\overline{z_{k}}f(x_{j}-x_{k})\geq0\,\,(resp.\,>0)\]
hold true.(see \rm{[10]}).\\
In 1930, the class of positive definite functions is fully characterized by Bochner's theorem \rm{[1]}, the function $f$ being positive definite if and only if it is the Fourier transform of a positive finite Borel measure $\mu$ on the real line $\mathbb{R}:$
$$f(x)=\int_{\mathbb{R}}e^{-itx}d\mu(t).$$\\
In \rm[7],  we have introduced the notion of Dunkl positive definite and strictly Dunkl positive definite functions on $\mathbb{R}^{d}$. We have established the analogue of Bochner's theorem in Dunkl setting.\\
A continuous function $f$ on $(a,b)$ is called completely monotonic on $(a,b)$,  if it  satisfies $f\in C^{\infty}(]a,b[)$ and 
$$(-1)^{n}f^{(n)}(x)\geq 0,$$
for all $n=0,1,2,...$ and $a<x<b$ (see \rm{[21]}).\\
Bernstein's Theorem \rm{[21, p. 161]}, states that a function $f:[0,\infty[\longrightarrow\mathbb{R}$ is completely monotonic
on $[0,\infty[$, if and only if 
$$f(x)=\int_{0}^{\infty}e^{-tx}d\mu(t)$$
where $\mu$ is a nonnegative finite Borel measure on $[0,\infty[.$\\
In 1938, Schoenberg's theorem \rm{[14]}, asserts that a function $\varphi$ is completely monotonic  on $[0,\infty[$ if and only if $\Phi:=\varphi(\parallel.\parallel^{2})$ is positive definite on every $\mathbb{R}^{d}.$\\
In this work, we introduce the notion of Dunkl completely monotonic functions on $\left(-\sigma,\sigma\right), \sigma>0.$ We establish the analogue of Schoenberg's theorem in Dunkl setting. As application we study the Dunkl complete monotonicity of the Kummer confluent hypergeometric functions.\\

\noindent Our paper is organized as follows: In section 2, we present some preliminaries results and notations that will be useful in the sequal. In section 3, we give some properties of the Dunkl kernel, the Dunkl transform and the Dunkl translation. In section 4, we recall some results about Dunkl positive definite functions proved by the authors in \rm{[7]}. In section 5, we introduce the notion of Dunkl completely monotonic functions in studying their properties, some examples are given. We state a  restrictive version of Schoenberg's theorem in Dunkl setting. As application, we study the Dunkl completely monotonicity of a class of functions related to the Kummer confluent hypergeometric functions .\\

\noindent Let us recall some classical functional spaces : 
\begin{itemize} 
\item $C(\mathbb{R}^{d})$ the set of continuous functions on $\mathbb{R}^{d},C_{0}(\mathbb{R}^{d})$ its subspace of continuous functions on $\mathbb{R}^{d}$ vanishing at infinity and $C^{\infty}(\mathbb{R}^{d})$ its subspace of infinitely differentiable functions.
\item $\mathcal{S}(\mathbb{R}^{d})$ the Schwartz space.
\item $L^{p}\left(\mathbb{R}^{d},h_{\kappa}^{2}\right),\;1\leq p<\infty$, the space of measurable functions on $\mathbb{R}^{d}$ such that 
\[\parallel f\parallel_{\kappa,p}=\left(\int_{\mathbb{R}^{d}}|f(x)|^{p}h_{\kappa}^{2}(x)dx\right)^{\frac{1}{p}}<\infty.\]
\item Let $\sigma>0, M_{\sigma}(\mathbb{R})$ denotes the space of nonnegative finite Borel measures on $\mathbb{R}$ satisfying
$$\int_{0}^{\infty}e^{\sigma\mid x\mid}d\mu(x)<\infty,$$
and $$\displaystyle  M_{+\infty}\left(\mathbb{R}\right)=\cap_{\sigma>0} M_{\sigma}\left(\mathbb{R}\right).$$
\end{itemize} 
\section{Notations and preliminaries}
Let $R$ be a fixed root system in $\mathbb{R}^{d}$, $G$ the associated finite reflexion group, and $R_{+}$ a fixed positive subsystem of $R,$ normalized so that $<\alpha,\alpha>=2$ for all $\alpha\in R_{+}$, where $<x,y>$ denotes the usual Euclidean inner product.\\
For a non zero $\alpha\in \mathbb{R}^{d}$, let use define the reflexion $\sigma_{\alpha}$ by 
\[\sigma_{\alpha}x=x-2\frac{<x,\alpha>}{<\alpha,\alpha>}\alpha,\,\, x\in\mathbb{R}^{d}.\]
Let $k$ be a nonnegative multiplicity function $\alpha\longmapsto k_{\alpha}$ defined on $R_{+}$ with the property that $k_{\alpha}=k_{\beta}$ where $\sigma_{\alpha}$ is conjugate to $\sigma_{\beta}$ in $G$. The weight function $h_{k}$ is defined by 
\begin{equation}
h_{k}(x)=\prod_{\alpha\in R_{+}}|<x,\alpha>|^{k_{\alpha}},\,\, x\in\mathbb{R}^{d}.
\end{equation}
This is a nonnegative homogeneous function of degre $\displaystyle{\gamma_{k}=\sum_{\alpha\in R_{+}}k_{\alpha}}$, which is invariant under the reflexion group $G.$\\
Let $T_{i}$ denote Dunkl's differential-difference operator defined in \rm{[4]} by 
\begin{equation}
T_{i}f(x)=\partial_{i}f(x)+\sum_{\alpha\in R_{+}}\kappa_{\alpha}\frac{f(x)-f(\sigma_{\alpha}x)}{<\alpha,x>}<\alpha,e_{i}>,\,\,1\leq i\leq d,
\end{equation}
where $\partial_{i}$ is the ordinary partial derivative with respect to $x_{i}$, and $e_{1},e_{2},...,e_{d}$ are the standard unit vectors of $\mathbb{R}^{d}.$\\
The rank-one cas: in cas $d=1$, the only choise of $R$ is $R=\{\pm\sqrt{2}\}$. The corresponding reflexion group is $G=\left\{ id, \sigma \right\}$ action on $\mathbb{R}$ by $\sigma(x)=-x.$ The Dunkl operator $T:=T_{k}$ associated with the multiplicity parameter $k\in\mathbb{C}$ is given by 
$$\displaystyle
T_{k}f(x)=f^{'}(x)+k\frac{f(x)-f(-x)}{x}.$$
Let $\mathcal{P}_{n}^{d}$ denote the space  of homogeneous polynomials of degree $n$ in $d-$variables. The operators $T_{i},\,1\leq i\leq d$ map  $\mathcal{P}_{n}^{d}$ to $\mathcal{P}_{n-1}^{d}.$\\
 The intertwining operator $V_{\kappa}$ is linear operator and determined uniquely as 
\begin{equation}
V_{\kappa}\mathcal{P}_{n}^{d}\subset\mathcal{P}_{n}^{d},\,\,V_{\kappa}1=1,\,\,\mathcal{T}_{i}V_{\kappa}=V_{\kappa}\partial_{i},\,1\leq i\leq d.
\end{equation}
According to R$\rm{\ddot{o}}$sler \rm{[13]}, $V_{k}$ is a positive operator. De Jeu \rm{[2]}, prouve that $V_{k}$ is an isomorphism of $C^{\infty}(\mathbb{R}^{d})$ whose inverse is denoted by $W_{k}$ and admit the following integral representation, 
\begin{theorem}\label{t1} For $f\in C(\mathbb{R}^{d}),$ we have 
\[
V_{k}f(x)=\int_{\mathbb{R}^{d}}f(y)d\mu_{x}(y),\;\;x\in\mathbb{R}^{d},\]
where $\mu_{x}$ is a probability measure on $\mathbb{R}^{d}$ which the carrier is in the closed ball $\overline{B(0,\parallel x\parallel)}.$
\end{theorem}
The Dunkl kernel associated with $G$ and $k$ is defined by \rm{[4]}: for $y\in\mathbb{C}^{n}$
$$E_{k}(x,y)=V_{k}\left(e^{<.,y>}\right)(x),\,\, x\in\mathbb{R}^{d}.$$
$$E_{k}(x,iy)=V_{k}\left(e^{<.,iy>}\right)(x),\,\, x,y\in\mathbb{R}^{d}.$$
plays the role of $e^{i<x,y>}$ in the ordinary Fourier analysis.\\
In the rank-one case: for the group $G=\mathbb{Z}_{2},\, Re(k)>0$ we have
$$\displaystyle
V_{k}f(x)=\frac{\Gamma(k+\frac{1}{2})}{\Gamma(\frac{1}{2})\Gamma(k)}\int_{-1}^{1}f(xt)(1-t)^{k-1}(1+t)^{k}dt.$$
In particular, for $x,y\in\mathbb{C}, Re(k)>0$ 
$$\displaystyle
E_{k}(x,y)=\frac{\Gamma(k+\frac{1}{2})}{\Gamma(\frac{1}{2})\Gamma(k)}\int_{-1}^{1}e^{xt}(1-t)^{k-1}(1+t)^{k}dt.$$

\begin{equation}\label{ja} 
E_{k}(x,y)=j_{k-\frac{1}{2}}(ixy)+\frac{xy}{(2k+1)}j_{k+\frac{1}{2}}(ixy)
\end{equation}
where for $\alpha\geq\frac{-1}{2}, j_{\alpha}$ is the normalized Bessel function.
\begin{propo}\rm{(see [11])}. 
Let $k\geq 0$ and $y\in\mathbb{C}^{d}$. Then the function $f=E_{\kappa}(.,y)$
is the unique solution of the system 
\begin{equation}
T_{i}f=<e_{i},y>f,\, \textrm{for\, all}\,\,1\leq i\leq d,
\end{equation}
which is real-analytic on $\mathbb{R}^{d}$ and satisfies $f(0)=1.$
\end{propo}
\begin{propo}\rm{(see [6,11])}. 
For $x,y\,\in\mathbb{C}^{d},\,\lambda\in\mathbb{C}$
\begin{enumerate}
\item $E_{\kappa}\left(x,\, y\right)=E_{\kappa}\left(y,\, x\right),$
\item $E_{\kappa}\left(\lambda x,\, y\right)=E_{\kappa}\left(x,\, \lambda y\right)$
\item $\overline{E_{\kappa}\left(x,\, y\right)}=E_{\kappa}\left(\overline{x},\,\overline{y}\right)$
\item $|E_{\kappa}(-ix,y)|\leq1$.
\item $\mid E_{\kappa}\left(x,\, y\right)\mid\leq e^{\parallel x\parallel.\parallel y\parallel},$ 
\end{enumerate}
\end{propo}
\section{Harmonic analysis related to the Dunkl operator}
\noindent In this section, we present some properties of the Dunkl transform, the Dunkl translation and the Dunkl convolution studied and developed in great detail in [2,6,16,18].\\
The Dunkl transform is defined for $f\in L^{1}\left(\mathbb{R}^{d},h_{\kappa}^{2}\right)$ by :
\begin{equation}
D_{\kappa}f(x)=c_{\kappa}\int_{\mathbb{R}^{d}}f(y)E_{\kappa}\left(-ix,\, y\right)h_{\kappa}^{2}(y)dy,\,\,\, x\in\mathbb{R}^{d}.
\end{equation}
If $\kappa=0$, then $V_{\kappa}=id$ and the Dunkl transform coincides with the usual Fourier transform. If $d=1$ and $G=\mathbb{Z}_{2}$, then the Dunkl transform is related closely to the Hankel transform on the real line.\\
\begin{theorem}\rm{(see [16])}.
\begin{enumerate}
\item For $f\in L^{1}\left(\mathbb{R}^{d},\, h_{k}^{2}\right),$ we have $D_{\kappa}f\in C_{0}\left(\mathbb{R}^{d}\right),$ and 
\[\parallel D_{\kappa}f\parallel_{C_{0}}\leq\parallel f\parallel_{\kappa,1}.\]
\item When both $f$ and $D_{\kappa}f$ are $\in L^{1}\left(\mathbb{R}^{d},\, h_{k}^{2}\right),$ we have the inversion formula 
$$\displaystyle f(x)=c_{\kappa}\int_{\mathbb{R}^{d}}D_{\kappa}f(y)E_{\kappa}(ix,\, y)h_{\kappa}^{2}(y)dy.$$
\item  The Dunkl transform $D_{\kappa}$ is an isomorphism of the Schwartz
class $\mathcal{S}(\mathbb{R}^{d})$ onto itself, and $D_{\kappa}^{2}f(x)=f(-x).$
\item The Dunkl transform $D_{\kappa}$ on $\mathcal{S}(\mathbb{R}^{d})$ extends uniquely to an isometry of $L^{2}\left(\mathbb{R}^{d},\, h_{k}^{2}\right).$
\item If $f,\, g\in L^{2}\left(\mathbb{R}^{d},\, h_{k}^{2}\right)$ then
$$\displaystyle \int_{\mathbb{R}^{d}}D_{\kappa}f(y)g(y)h_{\kappa}^{2}(y)dy=\int_{\mathbb{R}^{d}}f(y)D_{\kappa}g(y)h_{\kappa}^{2}(y)dy.$$
\end{enumerate}
\end{theorem}
\noindent Let $y\in\mathbb{R}^{d}$ be given. The Dunkl translation operator $f\longmapsto\tau_{y}f$ is defined in $L^{2}\left(\mathbb{R}^{d},\, h_{k}^{2}\right)$ by the equation 
\begin{equation}
D_{\kappa}(\tau_{y}f)(x)=E_{\kappa}(iy,x)D_{\kappa}f(x),\,\,\, x\in\mathbb{R}^{d}.
\end{equation}
The above definition gives $\tau_{y}f$ as an $L^{2}$ function.\\
Let 
\begin{equation}
A_{\kappa}(\mathbb{R}^{d})=\left\{ f\in L^{1}(\mathbb{R}^{d},h_{\kappa}^{2}):\, D_{\kappa}f\in L^{1}(\mathbb{R}^{d},h_{\kappa}^{2})\right\}.
\end{equation}
Note that $A_{\kappa}(\mathbb{R}^{d})$ is contained in the intersection of $L^{1}(\mathbb{R}^{d},h_{\kappa}^{2})$ and $L^{\infty}$ and hence is a subspace of $L^{2}(\mathbb{R}^{d},h_{\kappa}^{2})$. For $f\in A_{\kappa}(\mathbb{R}^{d})$ we have
\begin{equation}
\tau_{y}f(x)=\int_{\mathbb{R}^{d}}E_{\kappa}(ix,y)E_{\kappa}(-iy,\xi)D_{\kappa}f(\xi)h_{\kappa}^{2}(\xi)d\xi.
\end{equation}
Before stating some properties of the generalized translation operator let us mention that there is an abstract formula for $\tau_{y}$ given in terms of intertwining operator $V_{k}$ and its inverse. It takes the form \rm[18]. For $f\in C^{\infty}(\mathbb{R}^{d})$ we have 
\begin{equation}\label{ekm01} 
\tau_{y}f(x)=V_{k}^{(x)}\otimes V_{k}^{(y)}\left(W_{k}f(x-y)\right).
\end{equation}
\begin{theorem}\label{t3}\rm{(see[17])}. If $\varphi\in\mathcal{A}_{k}(\mathbb{R})$, then
$$\displaystyle
W_{k}\varphi(x)=\frac{1}{c_{k}}\int_{\mathbb{R}^{d}}e^{i<x,y>}D_{k}\varphi(y)h_{k}^{2}(y)dy.$$
\end{theorem}
\vskip0.3cm
\section{strictly Dunkl positive definite functions}
\begin{defin}\label{d1}
A function $\varphi$ of $ L^{2}\left(\mathbb{R}^{d},\, h_{k}^{2}\right)$ is called
Dunkl positive definite (resp. stictly Dunkl positive definte) if for  every finite distinct real numbers  $x_{1},...,x_{n},$ and every complex numbers $\alpha_{1}\,,...,\,\alpha_{n} $, not all zero, the inequality 
$$\displaystyle \sum_{j=1}^{n}\sum_{k=1}^{n}\alpha_{j}\overline{\alpha_{k}}\tau_{x_{j}}\left(\varphi\right)(x_{k})\geq0,\;\;(resp.>0)$$
holds true. 
Where $\tau_{x}$ denotes the Dunkl translation.
\end{defin} 
\begin{theorem}\rm{(see [7])}. 
Let $\varphi\in\mathcal{A}_{\kappa}(\mathbb{R}^{d})$, nonidentically zero and Dunkl positive definite function. Then $\varphi$ is strictly Dunkl positive definite.
\end{theorem}
\begin{theorem}\rm{(see [7])}. 
Let $\varphi\in \mathcal{A}_{\kappa}(\mathbb{R}^{d}).$ Then, $\varphi$ is Dunkl positive
definite, if and only if, there exist a nonnegative function $\psi\in \mathcal{A}_{\kappa}(\mathbb{R}^{d})$ such that
\begin{equation}
\varphi=D_{\kappa}\psi.
\end{equation}
\end{theorem}
\begin{defin}
A function $\Phi:\mathbb{R}^{d}\longrightarrow\mathbb{R}$ is said
to be radial if there exists a function $\varphi:[0,\infty[\longrightarrow\mathbb{R}$
such that $\Phi(x)=\varphi\left(\parallel x\parallel\right)$ for all
$x\in\mathbb{R}^{d}.$
\end{defin}
\vskip0.3cm
\section{Dunkl completely monotonic functions}

\begin{defin}\label{d3} A function $\varphi$ is called Dunkl completely monotonic on 
$(-\sigma,\sigma), \sigma> 0$ if $\varphi\in C\left((-\sigma,\sigma)\right)$ has 
derivatives for all orders on $]-\sigma,\sigma[$ and 
\begin{equation}
(-1)^{n}T_{k}^{n}\varphi(x)\geq0
\end{equation}
for all $n\in\mathbb{N}$ and $x\in]-\sigma,\sigma[.$\\
For $k=0, T_{k}f=f'$, we retreive the classical definition.
\end{defin}

\begin{remark} 
It's clear that if $\varphi$ and $\psi$ are Dunkl completely monotonic, then $\alpha\varphi+\beta\psi$
is also, where $\alpha$ and $\beta$ are a nonnegative constants.
\end{remark}
\begin{exe}\label{ex1}
For $y\geq0,$ the function $x\longmapsto E_{k}(-x,y)$ is Dunkl completely
monotonic on $\mathbb{R}.$ Indeed, for $x,y\in\mathbb{R}$ we have :
$$E_{k}(x,y)\geq 0$$
and 
$$T_{k}E_{k}(-x,y)=-yE_{k}(-x,y).$$
Thus
$$(-1)^{n}T_{k}^{n}E_{k}(-x,y)=y^{n}E_{k}(-x,y)\geq 0;\;y\geq 0.$$
\end{exe}
\begin{propo}\label{www}
Let $0<\sigma\leq+\infty$ and $\mu$ a  measure
on $M_{\sigma}(\mathbb{R}),$ then 
$$\displaystyle 
\varphi(x)=\int_{0}^{+\infty}E_{k}(-x,y)d\mu(y)$$
is Dunkl completely monotonic on $[-\sigma,\sigma].$
\end{propo}
\begin{proof} By example \ref{ex1} and since $\mu\in M_{\sigma}(\mathbb{R}),$ we get 
$$(-1)^{n}T_{k}^{n}\varphi(x)=\int_{0}^{\infty}y^{n}E_{k}(-x,y)d\mu(y)\geq 0$$
for all $n\in\mathbb{N}$ and $x\in]-\sigma,\sigma[.$ Moreover, $\varphi$
is continuous on $[-\sigma,\sigma],$ we conclude.
\end{proof}

\begin{propo}\label{p4}
Let $\varphi\in C^{\infty}\left((a,b)\right)$ and $\varphi$ is completely
monotonic function on $(a,b),$ then $V_{k}\varphi$ is Dunkl completely
monotonic on $(a,b).$
\end{propo}
\begin{proof}
Since $\varphi$ is completely monotonic on $(a,b),$ then 
$$(-1)^{n}\varphi^{(n)}(x)\geq0,\,\, a<x<b.$$
As $V_{k}$ is a positive operator and satisfies 
$$T_{k}\left(V_{k}\varphi\right)=V_{k}\left(\varphi^{'}\right)$$
We get 
$$(-1)^{n}T_{k}^{n}V_{k}\varphi(x)=V_{k}\left((-1)^{n}\varphi^{(n)}(x)\right)\geq0,\,\, a<x<b.$$
\end{proof}
\begin{propo}
Let $\varphi\in C^{1}\left((a,b)\right).$ If $\varphi$ is Dunkl
completely monotonic on $(a,b),$ then $-T_{k}\varphi$ is also Dunkl
completely monotonic on $(a,b).$
\end{propo}
\begin{proof}
Follows immediately by the definition \ref{d3}.
\end{proof}

\begin{theorem}
Let $\varphi\in\mathcal{A}_{k}(\mathbb{R})$ and $\mu$ is a  measure on $M_{\infty}(\mathbb{R})$ 
such that :
\begin{equation}\label{0017}
\varphi(x)=\int_{0}^{\infty}E_{k}(-x,y)d\mu(y),
\end{equation}
then, the function  $\varphi(x^{2})$ is Dunkl positive definite on $\mathbb{R}$ if and only if $\varphi(x)$ is Dunkl completely monotonic.
\end{theorem}
\begin{proof}
For $t>0$, the function $x\longmapsto f_{t}(x)=e^{-t^{2}x^{2}}$
is positive definite on $\mathbb{R}$. Bochner's theorem implies
that 
$$\displaystyle f_{t}(x)=\int_{\mathbb{R}}e^{-itxy}d\mu(y)$$
where $\mu$ is a finite nonnegative Borel measure on $\mathbb{R}.$\\
Then we have 
$$\displaystyle 
V_{k}\left(f_{t}\right)(x)=E_{k}(-t^{2},x^{2})=\int_{\mathbb{R}}\left(\int_{\mathbb{R}}e^{-i<zt,y>}d\mu(y)\right)d\mu_{x}(z)$$
Since the measures $\mu$ and $\mu_{x}$ are bounded, we have 
$$\displaystyle \phi_{t}(x)=E_{k}(-t^{2},x^{2})=\int_{\mathbb{R}}E_{k}(-itx,y)d\mu(y).$$
Now, we will prove that the function $\phi_{t}:x\longmapsto E_{k}(-t^{2},x^{2})$ is Dunkl positive definite on $\mathbb{R}$, for all $t>0.$ In fact from the formula (\ref{ekm01}) and theorem \ref{t3}, we have 
\begin{equation*}
\begin{split}
\tau_{x}\phi_{t}(y)&=\int_{\mathbb{R}}\int_{\mathbb{R}}W_{k}\phi_{t}(\eta-\xi)d\mu_{x}(\eta)d\mu_{y}(\xi)\\
&=\int_{\mathbb{R}}\int_{\mathbb{R}}W_{k}\left(V_{k}(f_{t})\right)(\eta-\xi)d\mu_{x}(\eta)d\mu_{y}(\xi)\\
&=\int_{\mathbb{R}}\int_{\mathbb{R}}f_{t}(\eta-\xi)d\mu_{x}(\eta)d\mu_{y}(\xi)\\
&=\int_{\mathbb{R}}\int_{\mathbb{R}^{d}}\int_{\mathbb{R}}e^{-i<t\eta,z>}e^{i<t\xi,z>}d\mu_{x}(\eta)d\mu_{y}(\xi)d\mu(z)\\
&=\int_{\mathbb{R}}\left[\int_{\mathbb{R}^{d}}e^{-i<t\eta,z>}d\mu_{x}(\eta)\right]\left[\int_{\mathbb{R}}e^{i<t\xi,z>}d\mu_{y}(\xi)\right]d\mu(z)\\
&=\int_{\mathbb{R}}E_{k}\left(-itx,z\right)E_{k}(ity,z)d\mu(z)\\
&=\int_{\mathbb{R}}\int_{\mathbb{R}}\overline{E_{k}(itx,z)}E_{k}(ity,z)d\mu(z),
\end{split}
\end{equation*}
\\
\noindent which implies that for every finite distinct real numbers $x_{1},x_{2},...,x_{n}$ and every complex numbers $\alpha_{1},\alpha_{2},...,\alpha_{n}$ not all zero, we get 
$$\displaystyle \sum_{j=1}^{n}\sum_{k=1}^{n}\alpha_{j}\overline{\alpha_{k}}\tau_{x_{j}}\phi_{t}(x_{k})=\int_{\mathbb{R}}
\left|\sum_{j=1}^{n}\alpha_{j}E_{k}(itx_{j},z)\right|^{2}d\mu(z)\geq 0.$$

\noindent Hence, the function $x\longmapsto\phi_{t}(x)$ is Dunkl positive definite on $\mathbb{R}$, for all $t>0$.\\
Next, we define the function  
$$\displaystyle  \Phi(x)=\int_{0}^{\infty}E_{k}(-x^{2},t^{2})d\mu(t).$$ 
\\
\noindent Since the function $\Phi\in\mathcal{A}_{k}(\mathbb{R})$, we have for every finite distinct real numbers $x_{1},x_{2},...,x_{n}$ and every complex numbers $\alpha_{1},\alpha_{2},...,\alpha_{n}$ not all zero 

$$\displaystyle
\sum_{j=1}^{n}\sum_{l=1}^{n}\alpha_{j}\overline{\alpha_{l}}\tau_{x_{j}}\Phi(x_{l})=\sum_{j=1}^{n}\sum_{l=1}^{n}\alpha_{j}\overline{\alpha_{l}}\int_{\mathbb{R}}E_{k}(-ix_{j},\xi)E_{k}(ix_{l},\xi)D_{k}\Phi(\xi)h_{k}^{2}(\xi)d\xi \qquad\qquad$$
$$\displaystyle \qquad\qquad=c_{k}\sum_{j=1}^{n}\sum_{l=1}^{n}\alpha_{j}\overline{\alpha_{l}}\int_{\mathbb{R}}E_{k}(-ix_{j},\xi)E_{k}(ix_{l},\xi)\left[\int_{\mathbb{R}}E_{k}(-is,\xi)\Phi(s)h_{k}^{2}(s)ds\right]h_{k}^{2}(\xi)d\xi$$
$$\displaystyle   =\int_{0}^{\infty}\sum_{j=1}^{n}\sum_{l=1}^{n}\alpha_{j}\overline{\alpha_{l}}\int_{\mathbb{R}}E_{k}(-ix_{j},\xi)E_{k}(ix_{l},\xi)\left[c_{k}\int_{\mathbb{R}}E_{k}(-is,\xi)\phi_{t}(s)h_{k}^{2}(s)ds\right]
\times h_{k}^{2}(\xi)d\xi d\mu(t)$$
$$\displaystyle =\int_{0}^{\infty}\sum_{j=1}^{n}\sum_{l=1}^{n}\alpha_{j}\overline{\alpha_{l}}\int_{\mathbb{R}}E_{k}(-ix_{j},\xi)E_{k}(ix_{l},\xi)D_{k}(\phi_{t})(\xi)h_{k}^{2}(\xi)d\xi d\mu(t)$$
$$\displaystyle =\int_{0}^{\infty}\sum_{j=1}^{n}\sum_{l=1}^{n}\alpha_{j}\overline{\alpha_{l}}\tau_{x_{j}}\phi_{t}(x_{l})d\mu(t)\geq 0.\qquad\qquad\qquad\qquad\qquad\qquad$$
The last inequality holds because the function $\phi_{t}(x)=E_{k}(-t^{2},x^{2})$ is Dunkl positive definite and the measure $\mu$ is nonnegative. Thus the function $\Phi$ is Dunkl positive definite. Since 

$$\displaystyle \varphi(x^{2})
=\int_{0}^{\infty}E_{k}(-t^{2},x^{2})d\mu(t)=\int_{0}^{\infty}E_{k}(-t,x^{2})d\mu(\sqrt{t})=\int_{0}^{\infty}E_{k}(-t,x^{2})d\nu(t)=\Phi(x),$$
we conclude.
\end{proof}

\begin{lemme}\label{l1}
Let $\varphi\in\mathcal{A}_{k}(\mathbb{R}^{d})$ be a radial function.
If $\varphi$ is Dunkl positive definite function then $D_{k}\varphi$
is even.
\end{lemme}
\begin{proof}
For $\varphi\in\mathcal{A}_{k}(\mathbb{R}^{d}),$ we have 
$$\displaystyle 
D_{k}\varphi(x)=c_{k}\int_{\mathbb{R}^{d}}E_{k}(-ix,y)\varphi(y)h_{k}^{2}(y)dy.$$
Thus 
\begin{equation*}
\begin{split}
D_{k}\varphi(-x)&=c_{k}\int_{\mathbb{R}^{d}}E_{k}(ix,y)\varphi(y)h_{k}^{2}(y)dy\\
&=\overline{\int_{\mathbb{R}^{d}}E_{k}(-ix,y)\overline{\varphi(y)}h_{k}^{2}(y)dy}\\
&=\overline{D_{k}\varphi(x)}.
\end{split}
\end{equation*}
Finally, corollary 1 in [7] completes the proof.
\end{proof}
\begin{lemme}\label{l2} Let $\varphi\in\mathcal{A}_{k}(\mathbb{R}).$ If $\varphi$ is Dunkl 
positive definite then the function $W_{k}\varphi$ is strictly positive definite on $\mathbb{R}$. 
\end{lemme}
\begin{proof}
For $\varphi\in\mathcal{A}_{k}(\mathbb{R}),$ by theorem \ref{t3}, we have 
\begin{equation}\label{1w1}
W_{k}\varphi(x)=\int_{\mathbb{R}}e^{ixy}D_{k}\varphi(y)h_{k}^{2}(y)dy.
\end{equation}
Since $\varphi$ is Dunkl positive definite function on $\mathbb{R},$
we obtain that the function $D_{k}\varphi$ is nonnegative. Thus,
for every finite distinct real numbers $x_{1},...,x_{n}$ and every
complex numbers $\alpha_{1},...,\alpha_{n}$ not all zero, we have 
$$\displaystyle  
\sum_{j=1}^{n}\sum_{l=1}^{n}\alpha_{j}\overline{\alpha_{l}}W_{k}\varphi(x_{j}-x_{l})=
\int_{\mathbb{R}}\bigg|\sum_{j=1}^{n}\alpha_{j}e^{ix_{j}y}\bigg|^{2}D_{k}\varphi(y)h_{k}^{2}(y)dy\geq0,$$
which implies that the function $W_{k}\varphi$ is positive definite
on $\mathbb{R}$.  Now, suppose that the function
$W_{k}\varphi$ is not strictly positive definite, then there exist distinct
reals points $x_{1},x_{2},...,x_{n}$ and complex numbers $\alpha_{1},\alpha_{2},...,\alpha_{n}$ not
all zero such that 
$$\displaystyle 
\sum_{j=1}^{n}\sum_{k=1}^{n}\alpha_{j}\overline{\alpha_{k}}W_{k}\varphi(x_{j}-x_{k})=0.$$
Thus 
$$\displaystyle  
\int_{\mathbb{R}}\left|\sum_{j=1}^{n}\alpha_{j}e^{ix_{j}t}\right|^{2}D_{k}\varphi(t)h_{k}^{2}(t)dt=0.$$
Since $\varphi$ is Dunkl positive definite and belongs to $\mathcal{A}_{k}(\mathbb{R})$, 
we have $D_{k}\varphi$ is nonnegative  continuous function. Then 
$$\displaystyle 
\left|\sum_{j=1}^{n}\alpha_{j}e^{i<x_{j},t>}\right|D_{k}\varphi(t)=0.$$
Moreover, since $D_{k}\varphi$ is nonidentically zero, then there exist
an open subset $U\subset\mathbb{R}$ such that 
$$\displaystyle 
D_{k}\varphi(t)\neq0,\quad \forall t\in U.$$
Thus
$$\displaystyle 
\sum_{j=1}^{n}\alpha_{j}e^{ix_{j}t}=0, \quad\forall t\in U.$$
From lemma 6.7 in \rm{[20~p.72]}, we get 
$$\displaystyle 
\alpha_{j}=0,\,\,\,\forall j\in\{1,...,n\}.$$
Then, we deduce that the function $W_{k}\varphi$ is strictly positive
definite function.
\end{proof}
 
\begin{theorem}
Let $\varphi\in\mathcal{A}_{k}(\mathbb{R})$ be a real function and
Dunkl positive definite. For $k>0$, we consider the function $\varphi_{k}(x)=
D_{k}\varphi(x) |x|^{2k+1}$.\\
If $\varphi_{k}$ is convex and verifie 
$\displaystyle \lim_{\mid x\mid\longrightarrow\infty}\varphi_{k}(x)=0,$
 then 
\begin{enumerate}
\item $W_{k}\varphi$ is even and nonnegative.
\item The function $\varphi(\sqrt{|x|})$ is Dunkl completely monotonic
on $\mathbb{R}.$
\end{enumerate}
\end{theorem}
\begin{proof}
 From lemma \ref{l1}, we conclude that the function $W_{k}\varphi$ is even and we have
$$\displaystyle 
W_{k}\varphi(x)=\int_{0}^{\infty}\cos(xy)D_{k}\varphi(y)|y|^{2k+1}dy.$$
Since the function $\varphi_{k}(y)=D_{k}\varphi(y)|y|^{2k+1}$ is
convex downwards on $[0,\infty[$ and \\$\displaystyle\lim_{\mid x\mid\longrightarrow\infty}\varphi_{k}(x)=0$, 
we deduce by lemma 1 in [22], that the function $W_{k}\varphi$ is nonnegative.\\
 By lemma 2,  the function $W_{k}\varphi$ is strictly positive definite on $\mathbb {R}$, nonnegative and radial. 
 From theorem 7.14 in [20], we conclude that the function $x\longmapsto W_{k}\varphi\left(\sqrt{|x|}\right)$ is completely monotonic on $\mathbb{R}$.  Proposition \ref{p4} completes the proof.
\end{proof}
\section{Applications}
\begin{theorem}
Let $p>0$ and $k>-1$. Put 
$$\displaystyle 
\varphi_{k,p}(x)=\frac{\Gamma(k+\frac{1}{2})e^{\frac{x^{2}}{4p}}}{2p^{k+\frac{1}{2}}}+\frac{\Gamma(k+1)x}{2(2k+1)p^{k+1}}\times_{1}F_{1}\left(k+1;k+\frac{3}{2};\frac{x^{2}}{4p}\right),$$
where $_{1}F_{1}$ is the Kummer confluent hypergeometric function. Then $\varphi_{k,p}$ is Dunkl completely monotonic on $\mathbb{R}.$
\end{theorem}
\begin{proof}
Let $d\mu(t)=e^{-pt^{2}}t^{2k+1}dt$ where $p>0,$ we obtain for all $\sigma>0$
$$\displaystyle \int_{0}^{\infty}e^{\sigma t}d\mu(t)=\int_{0}^{\infty}e^{t(\sigma-pt)}t^{2k+1}dt<+\infty,$$
which implies that the measure $\mu\in M_{\infty}(\mathbb{R})$. 
From the Sonine formula {[}19, p.394{]}, we have 
$$\displaystyle 
\int_{0}^{\infty}J_{k}(xt)e^{-pt^{2}}t^{k+1}dt=\frac{x^{k}e^{\frac{-x^{2}}{4p}}}{(2p)^{k+1}},$$
where $x,p,k$ complex numbers such that $Re(p)>0,\, Re(k)>-1$
and $J_{k}$ stands for the Bessel function of the first kind. We change in the above Sonine formula $x$ by $ix$, 
 we get :

\begin{equation}\label{kwh1}
\int_{0}^{\infty}j_{k-\frac{1}{2}}(ixt)e^{-pt^{2}}t^{k+\frac{1}{2}}dt=\frac{\Gamma(k+\frac{1}{2})e^{\frac{x^{2}}{4p}}}{2p^{k+\frac{1}{2}}}=I_{k,p}(x).
\end{equation}
On the other hand, using (\ref{ja}), we have 
$$\displaystyle \int_{0}^{\infty}E_{k}(-x,t)d\mu(t)=\int_{0}^{\infty}\left(j_{k-\frac{1}{2}}(ixt)+\frac{x}{2k+1}tj_{k+\frac{1}{2}}(ixt)\right)d\mu(t)=I_{k,p}(x)+J_{k,p}(x),$$
where 
$$\displaystyle 
J_{k,p}(x)=\frac{x}{2k+1}\int_{0}^{\infty}j_{k+\frac{1}{2}}(ixt)d\mu(t)=\frac{x}{2k+1}\int_{0}^{\infty}j_{k+\frac{1}{2}}(ixt)e^{-pt^{2}}t^{2k+1}dt.$$
Now, we calcul the function $J_{k,p}.$ From the integral representation 
$$\displaystyle 
\int_{0}^{\infty}t^{m+1}J_{k}(xt)e^{-pt^{2}}dt=\frac{x^{k}\Gamma(1+\frac{m}{2}+\frac{k}{2})}{2^{k+1}(\sqrt{p})^{k+m+2}\Gamma(k+1)}\times_{1}F_{1}\left(1+\frac{m}{2}+\frac{k}{2}:k+1;\frac{-x^{2}}{4p}\right).$$
We have 
$$\displaystyle 
\int_{0}^{\infty}t^{k+m+1}j_{k}(ixt)e^{-pt^{2}}dt=\frac{\Gamma(1+\frac{m}{2}+\frac{k}{2})}{2(\sqrt{p})^{k+m+2}}\times_{1}F_{1}\left(1+\frac{m}{2}+\frac{k}{2};k+1;\frac{x^{2}}{4p}\right)$$
Let $m=k-1$, we obtain :
$$\displaystyle 
\int_{0}^{\infty}t^{2k}j_{k}(ixt)e^{-pt^{2}}dt=\frac{\Gamma(k+\frac{1}{2})}{2(\sqrt{p})^{2k+1}}\times_{1}F_{1}\left(k+\frac{1}{2};k+1;\frac{x^{2}}{4p}\right).$$
Hence 
$$\displaystyle 
J_{k,p}(x)=\frac{\Gamma(k+1)x}{2(2k+1)p^{k+1}}\times_{1}F_{1}\left(k+1;k+\frac{3}{2};\frac{x^{2}}{4p}\right).$$
Finally, by Proposition \ref{www} we conclude that the function $\varphi_{k,p}$ is Dunkl completely monotonoic on $\mathbb{R}.$
\end{proof}
\begin{remark} For  $p>0$,  the function 
$$\phi_{p}(x)=\sqrt{\frac{\pi}{p}}e^{\frac{x^{2}}{4p}}+\frac{x}{p}\,_{1}F_{1}\left(1;\frac{3}{2};\frac{x^{2}}{4p}\right)$$
is  completely monotonic on $\mathbb{R}.$ In particular, for $p=\frac{1}{4}$, the function 
\begin{equation*}
\begin{split}
\phi_{0}(x)&=2\sqrt{\pi}e^{x^{2}}\left(1+\frac{\gamma(\frac{1}{2},x^{2})}{\sqrt{\pi}}\right)\\
&=2\sqrt{\pi}e^{x^{2}}\left(1+\erf(x)\right).
\end{split}
\end{equation*}
where $\gamma(a,z)$ and $\erf(z)$ are respectively the incomplete gamma and error functions, is  completely monotonic on $\mathbb{R}$.
\end{remark}

\end{document}